\documentclass{amsart}[12pt]
\usepackage{amsmath,amsfonts,amssymb,latexsym,amsthm}
\usepackage{epsfig}
\usepackage{hyperref} 
\usepackage[dvipsnames]{xcolor}

\numberwithin{equation}{section}

\newtheorem{thm}{Theorem}[section]
\newtheorem{lemma}[thm]{Lemma}

\newtheorem{cor}[thm]{Corollary}
\newtheorem{prop}[thm]{Proposition}
\newtheorem{conj}[thm]{Conjecture}

\newtheorem{rem}[thm]{Remark}
\newtheorem{ex}[thm]{Example}

\newcommand{\LR}{\rm{LR}}

\title[Upper bounds on Kronecker coefficients]
{Upper bounds on Kronecker coefficients with few rows}

\author[Igor Pak and Greta Panova]{Igor Pak$^\star$ \ \ and \ \ Greta Panova$^\dagger$}

\thanks{\today}
\thanks{\thinspace ${\hspace{-.45ex}}^\star$Department of Mathematics,
UCLA, Los Angeles, CA~90095.
\hskip.06cm
Email:
\hskip.06cm
\texttt{pak@math.ucla.edu}}

\thanks{\thinspace ${\hspace{-.45ex}}^\dagger$Department of Mathematics,
 USC, Los Angeles, CA~90089.
\hskip.06cm
Email:
\hskip.06cm
\texttt{gpanova@usc.edu}}

\newcommand{\SSYT}{\operatorname{SSYT}}

\def\bx{{\mathbf{x}}}
\def\by{{\mathbf{y}}}
\def\bz{{\mathbf{z}}}

\def\emp{\varnothing}

\def\nn{\mathbb N}

\def\rr{\mathbb R}

\def\la{\lambda}
\def\ga{\gamma}
\def\si{\sigma}

\def\eps{t}
\def\al{\alpha}
\def\be{\beta}
\def\om{\omega}

\def\ve{\varepsilon}

\def\cA{\mathcal A}

\def\cL{\mathcal L}

\def\cP{\mathcal P}

\def\CT{\mathcal T}
\def\BCT{\mathcal B}
\def\PCT{\mathcal Pyr}
\def\RCT{\mathcal P}
\def\QCT{\mathcal Q}

\def\ssu{\subset}

\def\<{\langle}
\def\>{\rangle}

\def\rE{ {\text {\rm E} } }
\def\rR{ {\text {\rm R} } }

\def\rT{{\text {\rm T} } }

\def\rB{{\text {\rm B} } }
\def\rP{{\text {\rm Pyr} } }

\def\dom{\unlhd}
\def\rdom{\unrhd}

\def\0{{\mathbf 0}}

\def\SS{{\mathbb{S}}}
\def\MM{{\mathbb{M}}}

\def\SP{{\textsc{\#P}}}
\def\NP{{\textsc{NP}}}

\def\.{\hskip.06cm}
\def\ts{\hskip.03cm}

\def\nin{\noindent}

\def\rg{\overline{g}}

\begin{document}

\begin{abstract}
We present three different upper bounds for Kronecker
coefficients $g(\la,\mu,\nu)$ in terms of Kostka numbers,
contingency tables and Littlewood--Richardson coefficients.
We then give various examples, asymptotic applications, and
compare them with existing lower bounds.
\end{abstract}

\maketitle

\section{Introduction} \label{sec:Intro}
%
\nin
Combinatorics is an eternally vibrant and rapidly growing field
of mathematics with a number of distinct areas featuring
fundamentally different problems, ideas, tools, goals and
techniques.  This remarkable diversity can, in principle, lead
to miscommunication, confusion, and a surprising
lack of understanding, but it can also be extremely
beneficial both mathematically in terms of inter-area
work, and metamathematically in terms of different
ways to frame a problem and formulate the answer.

In this paper we present a number of new upper bounds
on \emph{Kronecker coefficients}.  We employ some remarkable
upper bounds on \emph{contingency tables} combined with
our earlier bounds on  Kronecker and
\emph{Littlewood--Richardson} (LR--)
\emph{coefficients}, as well as other tools.
There are two distinct motivations behind our work.
First, the Kronecker coefficients are famously difficult
and mysterious, full of open problems such as the
\emph{Saxl Conjecture}.  This means that there are
very few strong results and those that exist are not very
general, so our general bounds can prove helpful in
applications.

More importantly, the Kronecker coefficients are famously
$\SP$-hard to compute, and $\NP$-hard to decide if they are
nonzero~\cite{IMW,Nar}, so one should not expect a closed
formula.  What makes the matters worse, it is a long standing
open problem~\cite{Stanley-kron} to find a combinatorial
interpretation for Kronecker coefficients, so it is not
even clear \emph{what}  we are counting.  Thus, good general
bounds is the next best thing one could hope for.

\smallskip

Recall that the \emph{Kronecker coefficients} \ts $g(\la,\mu,\nu)$ \ts
are defined as structure constants in products of $S_n$-characters:
$$
\chi^{\mu}\cdot \chi^\nu \. = \. \sum_{\la \vdash n} \. g(\la,\mu,\nu) \. \chi^\la\,,
$$
where \ts $\la,\mu,\nu\vdash n$ (see~$\S$\ref{sec:Basic} for the background).

\medskip

\begin{thm}[{= Theorem~\ref{t:Kron-3dim-rows}}]
Let $\la,\mu,\nu\vdash n$ such that $\ell(\la)=\ell$, $\ell(\mu) = m$,
and $\ell(\nu)=r$.  Then:
$$g(\la,\mu,\nu) \, \le \, \left(1+\frac{\ell m r}{n}\right)^n
\left(1+\frac{n}{\ell m r}\right)^{\ell m r} .
$$
\end{thm}

\smallskip

\nin
This is perhaps the cleanest and the most attractive upper bound of all bounds
we present.
In particular, when \ts $\ell m r \le n$, we have $g(\la,\mu,\nu) \le 4^n$,
which is often quite sharp compared to the only general upper bound \ts
$g(\la,\mu,\nu) \le \min\{f^\la,f^\mu,f^\nu\}$, see~$\S$\ref{ss:Basic-Kron}.

\medskip

The bounds we present are split into the following three approaches:

\smallskip

$\circ$ \. via Kostka numbers and $2$-dimensional contingency tables ($\S$\ref{sec:2CT}),

$\circ$ \. via $3$-dimensional contingency tables ($\S$\ref{sec:3CT}), and

$\circ$ \. via Vallejo's multi--LR coefficients and the inverse Kostka numbers
($\S$\ref{sec:multi-LR}).

\smallskip

\nin
The advantage of these approaches is the availability of
both exact and asymptotic upper bounds for all ingredients
(notably, see $\S$\ref{sec:CT-count} for an extensive
discussion on counting \emph{contingency tables}).
While not always comparable, they give
roughly similar results in some examples, leaving room for
improvement in various cases, especially in the lower order
terms which we intentionally do not optimize.

We also apply the contingency tables technology to obtain the
upper bounds for the \emph{reduced Kronecker coefficients}.
This is done via remarkable recent identity by Briand
and Rosas~\cite{BR}, see $\S$\ref{sec:Reduced}.

In the last part of the paper,
we compare our upper bounds with the upper and lower bounds coming
from counting \emph{binary contingency tables}, see~$\S$\ref{sec:3CT-bin}
and~$\S$\ref{sec:3CT-pyramids}.  Let us single out one curious lower bound:

\smallskip

\begin{cor} [{= Corollary~\ref{c:Kron-sym-sum}}]
Let \ts $\cL_n = \{\la\vdash n, \ts \la=\la'\}$.  We have:
$$
\sum_{\la\in \cL_n} \. g(\la,\la,\la) \, \ge \,
e^{c\ts n^{2/3}} \ \ \text{for some} \ \ c>0\ts.
$$
\end{cor}

We conclude the paper with final remarks and
open problems in~$\S$\ref{sec:fin-rem}.

\bigskip

\section{Basic definitions, results and notation} \label{sec:Basic}

\subsection{Partitions and Young tableaux} \label{ss:Basic-part}
We use standard notation from~\cite{Mac} and~\cite[$\S$7]{EC2}
throughout the paper.

Let $\la=(\la_1,\la_2,\ldots,\la_\ell)$ be a \emph{partition}
of size $n:=|\la|=\la_1+\la_2+\ldots+\la_\ell$, where
$\la_1\ge \la_2 \ge \ldots \ge\la_\ell\ge 1$.  We write
$\la\vdash n$ for this partition, and $\cP=\{\la\}$ for
the set of all partitions.   The length of~$\la$
is denoted $\ell(\la):=\ell$. Denote by $p(n)$ the number
of partitions $\la\vdash n$. Let $\la+\mu$ denotes a partition
$(\la_1+\mu_1,\la_2+\mu_2,\ldots)$

Special partitions include the \emph{rectangular shape} \ts
$(a^b) = (a,\ldots,a)$, $b$ times, the \emph{hooks shape}
$(k,1^{n-k})$, the \emph{two-row shape} $(n-k,k)$, and the
\emph{staircase shape} \ts $\rho_\ell = (\ell,\ell-1,\ldots,1)$.

A \emph{Young diagram} of
\emph{shape}~$\la$ is an arrangement of squares
$(i,j)\ssu \nn^2$ with $1\le i\le \ell(\la)$
and $1\le j\le \la_i$.  Let $\la,\mu\vdash n$.
A \emph{semistandard Young tableau} $A$ of
\emph{shape}~$\la$ and \emph{weight}~$\mu$ is an
arrangement of $\mu_k$ integers~$k$ in squares
of~$\la$, which weakly increase along rows and strictly
increase down columns.  Denote by $\SSYT(\la,\mu)$ the
set of such tableaux, and \ts
$K(\la,\mu) = \bigl|\SSYT(\la,\mu)\bigr|$ \ts the
\emph{Kostka number}.

A \emph{plane partition} $A=(a_{ij})$ of $n$ is an
arrangement of integers $a_{ij} \ge 1$ of a Young diagram
shape which sum to~$n$ and weakly decrease along rows
and columns. Denote by $p_2(n)$ the total number of such
plane partitions.

\subsection{{Representations of $S_n${\ts}}} \label{ss:Basic-reps}
We denote by $\SS^\la$ the irreducible representation of $S_n$
corresponding to partition $\la\vdash n$, and by $\chi^\la$ the
corresponding character.  Let
$$f^\la \. := \. \dim \SS^\la \. = \. \chi^\la(1) \. = \. K(\la,1^n)\ts.
$$
The \emph{hook-length formula} (HLF) is an explicit product formula for~$f^\la$,
see e.g.~\cite{Mac,EC2}.

Denote by \ts $\MM^\nu :={\rm Ind}_{S(\nu)}^{S_n} 1$ \ts
the \emph{induced representation}, where \ts
$S(\nu) = S(\nu_1)\times S(\nu_2) \times \ldots$, $\nu\vdash n$.
Denote by $\phi^\nu$ the corresponding character.
Then
$$
\phi^\nu(1)\. = \. \dim \MM^\la \. = \. \binom{n}{\nu_1\ts,\ts \nu_2\ts,\ts \ldots}
$$
and
$$
\phi^\nu \. = \. \sum_{\la\vdash n} \. K(\la,\nu) \. \chi^\la\,.
$$
The \emph{Littlewood--Richardson {\rm (LR-)} coefficients} \ts $c^\la_{\mu\nu}$ \ts
are defined as follows:
$$\chi^{\mu\circ\nu} \. = \. \sum_{\la \vdash n} \. c^\la_{\mu\nu} \. \chi^\la\,,
$$
where \ts $\chi^{\mu\circ\nu}$ \ts is the
character of the induced representation \ts $ \displaystyle {\rm Ind}_{S_k \times S_{n-k}}^{S_n} \SS^\mu \times \SS^\nu$,
and \ts $\la\vdash n$, \ts $\mu\vdash k$, \ts $\nu\vdash n-k$.

\medskip

\subsection{Kronecker coefficients} \label{ss:Basic-Kron}
As in the introduction, the \emph{Kronecker coefficients} \ts $g(\la,\mu,\nu)$ \ts
are defined as follows:
$$\chi^{\mu}\cdot \chi^\nu \. = \. \sum_{\la \vdash n} \. g(\la,\mu,\nu) \. \chi^\la\,,
$$
where \ts $\la,\mu,\nu\vdash n$.  Equivalently,
$$
g(\la,\mu,\nu) \. = \, \frac{1}{n!} \. \sum_{\si \in S_n} \.  \chi^\la(\si) \ts
\chi^\mu(\si) \ts \chi^\nu(\si) \ts.
$$
From here it is easy to see that:
\begin{equation}\label{eq:Kron-sym}
g(\la,\mu,\nu) \. = \. g(\mu,\la,\nu) \. = \. g(\la,\nu,\mu) \. = \. \ldots
\end{equation}
and
\begin{equation}\label{eq:Kron-transp}
g(\la,\mu,\nu) \. = \. g(\la',\mu',\nu)\ts.
\end{equation}

Also, for all \ts $f^\nu\le f^\mu\le f^\la$ \ts
we have:
\begin{equation}\label{eq:Kron-dim-upper}
g(\la,\mu,\nu) \, \le \, \frac{f^{\mu} f^{\nu}}{f^\la} \, \le \, f^\nu,
\end{equation}
see e.g.~\cite[Ex.~4.12]{Isa} and~\cite[Eq.~$(3.2)$]{PPY}.

The \emph{Saxl conjecture}, see~\cite{PPV},  states that \ts
$g\bigl(\rho_\ell,\rho_\ell,\nu) \ts \ge \ts 1$ \ts
for all \ts $\nu\vdash n=|\rho_\ell| = \ell(\ell+1)/2$,
where \ts
$\rho_\ell=(\ell,\ell-1,\ldots,1)$ \ts is the \emph{staircase shape}.
\begin{ex} {\rm
Let $n=\ell^3$, $k=\ell^2$, $\la=\mu=(k^\ell)$,
and $\ell\to \infty$.  The HLF gives:
$$
f^{\la}\. =\. f^{\mu} \, = \, \exp\.\bigl[\ell^3\log \ell + O(\ell^2) \bigr].$$
Similarly, in the case $\nu = (n/r)^r$, $r=O(1)$, we have
$$
f^\nu \, = \, \exp\.\bigl[n \log r + o(n)\bigr] \, = \,
\exp \Theta\bigl(\ell^3\log r\bigr)\ts.
$$
In fact, the upper bounds on $g(\la,\mu,\nu)$ implied
by~\eqref{eq:Kron-dim-upper} are very far from being tight.
For example, for $r=2$, we have:
$$
g(\la,\mu,\nu) \. \le \. \binom{k+\ell}{\ell} \. = \. \exp O(\ell\log \ell)\ts.
$$
See~\cite{MPP,PP2} for substantially better lower and upper bounds in this case.}
\label{ex:motiv-HLF}
\end{ex}

\bigskip

\section{Contingency tables} \label{sec:CT-count}

\subsection{Definition}\label{ss:CT-count-def}
Let $\mathbf a=(a_1,\dots,a_m)$
and $\mathbf b=(b_1,\dots,b_n)$,
be two integer sequences
with equal sum:
$$
\sum_{i=1}^m \ts a_i \, = \, \sum_{j=1}^n \ts b_j \. = \. N.
$$
A \emph{contingency table} with \emph{margins} $(\mathbf a, \mathbf b)$
is an $m \times n$ matrix of non-negative integers whose $i$-th row sums to $a_i$
and whose $j$-th column sums to~$b_j$. We denote by $\CT(\mathbf a, \mathbf b)$ the set of all
such matrices, and let \ts $\rT(\mathbf a, \mathbf b):=|\CT(\mathbf a, \mathbf b)|$.
Finally, denote by \.
$\RCT(\mathbf a, \mathbf b) \in \rr^{m \ts n}$ \. be the polytope of \emph{real}
contingency tables, i.e.\ table with row and column sums as above,
and non-negative real entries.

Counting \ts $\rT(\mathbf a, \mathbf b)$ \ts is a difficult problem, both mathematically
and computationally.  In fact, even a change in a single row and column sum can
lead to a major change in the count, see~\cite{B2,DLP}.  Three-dimensional tables
are even harder to count, see~\cite{B3,DO}.  We refer to~\cite{B2,DG,FLL}
for an introduction to the subject and further references in many areas.

Note that \ts $\rT(\mathbf a, \mathbf b)$ \ts is invariant under permutation of
the order of margins.  For simplicity of notation, throughout the paper we
use partitions to denote the margins.

\medskip

\subsection{Bounds for 2-dimensional tables}
Let $\la,\mu \vdash n$, $\ell=\ell(\la)$, $m=\ell(\mu)$.  Denote by
\. $\CT(\la,\mu)$ \. the set of $\ell\times m$ \emph{contingency tables}
with row sums $\la$ and column sums~$\mu$.  Let \.
$\rT(\la,\mu) = \bigl|\CT(\la,\mu)\bigr|$.

\medskip

\begin{thm} \label{t:barv-2dim}
Let \ts $\ell=\ell(\la)$, $m=\ell(\mu)$.
Let \. $Z=(z_{ij}) \in \RCT(\la,\mu)$ \. be the
unique point maximizing a strictly concave function
$$g(Z) \. := \, \sum_{i=1}^\ell\sum_{j=1}^m \. (z_{ij}+1) \log (z_{ij}+1)
\. - \. z_{ij}\log z_{ij}
$$
Then:
$$\rT(\la,\mu) \. \le \. \exp g(Z)\ts.
$$
\end{thm}

\smallskip

\begin{rem}{\rm
In~\cite{Sha} (see also~\cite[$\S$3]{B3}), an exponential
improvement in the upper bound was obtained.  Unfortunately,
it does not seem to improve our estimates except for the lower
order terms.  }
\end{rem}

\smallskip

\begin{ex}{\rm
Let $n=\ell^3$, $k=\ell^2$, $\la=\mu=(k^\ell)$,
and $\ell\to \infty$.  Note that contingency tables in $\CT(\la,\mu,\nu)$ all
have equal margins.  Thus, the convex polytope $\RCT(\la,\mu,\nu)$ is symmetric
with respect to $S_\ell\times S_\ell$ action, and since the  $Z\in\RCT(\la,\mu,\nu)$
maximizing $g(Z)$ is unique it must be uniform. In other words, $z_{ij} = \ell$
for all $1\le i,j\le \ell$.  This gives
$$
g(Z) \, = \, \ell^2\bigl[(\ell+1)\log (\ell+1) \. - \. \ell\log \ell\bigr]
\, = \, \ell^2\log \ell \. +\. \ell^2 \. + \. \frac12 \ts \ell \. + \. O(1)\ts.
$$
and
$$
\rT(\la,\mu) \, \le \, \exp\ts g(Z) \, = \, \exp \bigl[\ell^2\log \ell \. + \. O(\ell^2)\bigr]\..
$$
}\label{ex:motiv-2dim}
\end{ex}

\smallskip

We should mention that in the uniform case and $\ell=n^\ve$, $\ve<1/3$,
very precise asymptotics are known.  We will not use the matching
lower bounds and only include one such result in a somewhat simplified
form.

\smallskip

\begin{thm}[{Cor.~1 in \cite{CM}}]
Let $\la=(k^\ell)$, $\mu=(s^m)$, so $\ell \ts k = m \ts s = n \to \infty$.  Let
$\al=s/\ell=k/m$, s.t.\ $\ell\ts m = o(\al^2)$.   Then
$$
\rT(\la,\mu) \, = \, (\al+1/2)^{(\ell-1)(m-1)} \, \frac{(\ell \ts m)!}{(\ell!)^m \. (m!)^\ell}
\,\cdot\. O(1)\ts.
$$
\end{thm}

\nin
See also~\cite{BH1,GM} for  more general bounds in the
near-uniform case.

\medskip

\subsection{Bounds for 3-dimensional tables}
Let $\la,\mu,\nu\vdash n$.  Denote by $\rT(\la,\mu,\nu)$
the number of $3$-dimensional \ts $\ell(\la)\times \ell(\mu)\times \ell(\nu)$ \ts
contingency tables with $2$-dimensional sums orthogonal to $x, y$ and~$z$ coordinates
given by $\la$, $\mu$ and $\nu$, respectively.  Denote by $\RCT(\la,\mu,\nu)$
the corresponding polytope of real $3$-dimensional contingency tables.

\begin{thm} [{Barvinok~\cite[$\S3$]{B3} and Benson-Putnins~\cite{Ben}}]
Let \ts $\ell=\ell(\la)$, $m=\ell(\mu)$, $r=\ell(\nu)$.
Let \. $Z=(z_{ijk}) \in \RCT(\la,\mu,\nu)$ \. be the
unique point maximizing a strictly concave function
$$g(Z) \. := \, \sum_{i=1}^\ell\sum_{j=1}^m \sum_{k=1}^r \. (z_{ijk}+1) \log (z_{ijk}+1)
\. - \. z_{ijk}\log z_{ijk}\..
$$
Then:
$$\rT(\la,\mu,\nu) \. \le \. \exp g(Z)\ts.
$$
\label{t:barv-3dim}
\end{thm}

This result will prove quite sharp and allows us to bound Kronecker
coefficients in the rectangular case.

\medskip

\subsection{Bounds for binary tables}
Denote by  $\BCT(\la,\mu,\nu)$ the set of
$3$-dimensional \emph{binary} (0/1)
contingency tables, and let
$\rB(\la,\mu,\nu)=\bigl|\BCT(\la,\mu,\nu)\bigr|$.
Denote by
$$\QCT(\la,\mu,\nu) \. := \. \RCT(\la,\mu,\nu) \. \cap_{ijk} \.
\bigl\{0\le z_{ijk} \le 1\bigr\}
$$
the intersection of the polytope of contingency tables
with the unit cube.

\begin{thm} [{Barvinok, see e.g.\ \cite[$\S3$]{B2}}]
Let \ts $\ell=\ell(\la)$, $m=\ell(\mu)$, $r=\ell(\nu)$.
Let \. $Z=(z_{ijk}) \in \QCT(\la,\mu,\nu)$ \. be the
unique point maximizing a strictly concave function
$$h(Z) \. := \, \sum_{i=1}^\ell\sum_{j=1}^m \sum_{k=1}^r \. z_{ijk}\ts \log \frac{1}{z_{ijk}}
\, + \, (1-z_{ijk})\ts \log \frac{1}{1-z_{ijk}}\,.
$$
Then:
$$\rB(\la,\mu,\nu) \. \le \. \exp h(Z)\ts.
$$
\label{t:barv-3dim-binary}
\end{thm}

\smallskip

\begin{ex}{\rm
Let $n=k^3$, $\ell=k^2$, $\la=\mu=(k^\ell)$,
and $k\to \infty$.   Consider $\rB(\la,\mu)$.  By the symmetry,
$z_{ij} = 1/\ell$ for all $1\le i,j\le \ell$.  This gives
$$
h(Z) \. = \. \ell^2\cdot \left[\frac{1}{\ell}\log\ell \. + \.
\bigl(1-1/\ell\bigr)\log \frac{1}{1-1/\ell}\right] \. = \. \ell\log \ell +\ell+ O(1)
$$
and
$$
\rB(\la,\mu) \, \le \, \exp \bigl[\ell\log \ell \. + \. O(\ell)\bigr]\.,
$$
which is also tight~\cite{B2}.
}\label{ex:motiv-2dim-bin}
\end{ex}

\medskip

\subsection{Majorization}
Let $\la,\mu\vdash n$. The \emph{dominance order} is defined as follows: \.
$\la \unlhd \mu$ if $\la_1 \le \mu_1$, $\la_1+\la_2 \le \mu_1+\mu_2$, etc.
For $\la \vdash n$ and a set of partitions $\cL$,
we write $\la\unlhd \cL$ if $\la \dom \mu$ for all $\mu\in \cL$.
This is a special case of \emph{majorization}, equivalent for partitions
and studied extensively in many fields of mathematics and applications,
see e.g.~\cite{MOA}.  The followiiing result is standard in the area
(see e.g.~\cite{Mac,EC2}):

\smallskip

\begin{thm} \label{t:Kostka-dom}
Let $\la,\mu\vdash n$. Then \ts $K(\nu,\la)\ge K(\nu,\mu)$ \ts
for all \ts $\la \unlhd \mu$.  Moreover, we have
$K(\la,\la)=1$, and $K(\la,\mu)=0$ unless $\mu \dom \la$.
\end{thm}

\smallskip

We refer to~\cite[$\S$1.7]{Mac} for an algebraic proof,
to~\cite{Whi} for a direct bijective proof, and
to~\cite{Pak} for the context and generalizations.

\begin{thm}[Barvinok]\label{t:CT-dom-2dim}
Let \ts $\la,\mu,\al,\be \vdash n$, and suppose \ts $\ell(\la)=\ell(\al)$,
\ts $\ell(\mu) = \ell(\be)$, \ts $\la \rdom \al$, \ts $\mu \rdom \be$. Then:
$$
\rT(\la,\mu) \, \le \, \rT(\al,\be)\ts.
$$
\end{thm}

The proof in~\cite[Eq.~(2.4)]{B0} is a one line application of
Theorem~\ref{t:Kostka-dom} to the \emph{RSK identity} (see e.g.~\cite{EC2}):
$$
\rT(\la,\mu) \, = \, \sum_{\nu \vdash n} \. K(\la,\nu) \cdot K(\mu,\nu)
\, \le \, \sum_{\nu \vdash n} \. K(\al,\nu) \cdot K(\be,\nu) \, = \, \rT(\al,\be)\ts.
$$
Alternatively, it can be deduced from \cite[Thm.~4.9]{Val-dia}.
We refer~\cite{Pak} (Note~36,~37 in the expanded version on the paper),
for a explicit combinatorial proof.  The following in a helpful
extension of Theorem~\ref{t:CT-dom-2dim}.

\begin{thm}
Let \ts  $\la,\mu,\nu,\al,\be,\ga \vdash n$, and suppose \ts $\ell(\la)=\ell(\al)$,
\ts $\ell(\mu) = \ell(\be)$, \ts $\ell(\nu) = \ell(\ga)$, \ts $\la \rdom \al$, \ts $\mu \rdom \be$, \ts
\ts $\nu \rdom \ga$. Then:
$$
\rT(\la,\mu,\nu) \, \le \, \rT(\al,\be,\ga)\ts.
$$
\label{t:CT-dom-3dim}
\end{thm}

\begin{proof}  For a contingency table \ts $T\in \CT(\la,\mu,\nu)$, let
\ts $A=A(T)\in\CT(\mu,\nu)$  \ts be a partition of the $1$-margins,
i.e.\ of sums along lines parallel to $x$~axis.
Thus, projecting along the $x$ axis and applying Theorem~\ref{t:CT-dom-2dim},
we have:
$$
\rT(\la,\mu,\nu) \, = \, \sum_{A\in \CT(\mu,\nu)} \. \rT(\la,A)
\, \le \, \sum_{A\in \CT(\mu,\nu)} \. \rT(\al,A)
\, = \, \rT(\al,\mu,\nu)\ts.
$$
Applying this two more times, we obtain:
$$
\rT(\la,\mu,\nu) \, \le \, \rT(\al,\mu,\nu)\, \le \, \rT(\al,\be,\nu)
\, \le \, \rT(\al,\be,\ga)\ts,
$$
as desired.
\end{proof}

\begin{rem}{\rm Theorem~\ref{t:CT-dom-3dim} can be easily generalized
to $d$-dimensional contingency tables.  The proof by induction
follows verbatim the proof above.
}\end{rem}

\bigskip

\section{Kostka numbers approach} \label{sec:2CT}

\subsection{Bounds on Kostka numebrs}  \label{ss:2CT-Kostka-bounds}
We start with the following easy but useful bounds:

\smallskip

\begin{lemma}\label{l:Kostka-CT} For all $\la,\mu\vdash n$ we have:
$$
K(\la,\mu) \, \le \, \rT(\la,\mu) \quad \ \text{and} \quad K(\la,\mu) \, \le \, \rB(\la',\mu)\ts.
$$
\end{lemma}

\begin{proof}
For the first inequality, observe that every tableau $T\in \SSYT(\la,\mu)$
is encoded by an $\ell\times m$ array $X_T=(x_{ij})$, where $x_{ij}$
is the number of $i$-s in $j$-th row.  Observe that $X_T \in \CT(\la,\mu)$
by the definition of $\SSYT(\la,\mu)$.  Thus, $T\to X_T$ is an
injection, which implies the claim.

For the second inequality,
notice that in the same encoding of $T'\in \SSYT(\la',\mu)$ the resulting
$X_T$ is binary.  This follows from the fact that tabelau~$T$ is strictly
increasing in columns.
\end{proof}

\begin{ex}{\rm
let $\la=\mu\vdash n$.  It is easy to see that $K(\la,\mu)=1$.
On the other hand, $\rT(\la,\mu)$ can be quite large.  For example,
for $\la=\mu=(k,k)$, $n=2k$, we have $\rT(\la,\mu) = k+1$.
For $\la=\mu=1^n$, we have $\rT(\la,\mu) = n!$, while $\rB(\la,\mu)=1$
gives a sharp bound.

When $\la=\la'$ be a self-conjugate partition,
$\rB(\la',\mu)=\rB(\la,\mu) \ssu \rT(\la,\mu)$, giving often
a better bound. In the case $\la=\mu=(\ell^\ell)$, $n=\ell^2$,
we have $K(\la,\mu)=\rB(\la,\mu)=1$ as $\RCT(\ell^\ell,\ell^\ell)$ in
this case is an all-one $\ell\times\ell$ array.
}\label{ex:Kostka-basic}
\end{ex}

\medskip

\subsection{Upper bound}
Let $\la,\mu,\nu \vdash n$, $\ell=\ell(\la)$, $m=\ell(\mu)$.
We somewhat extend the notation as follows.  For $A=(a_{ij})\in \CT(\la,\mu)$,
denote by \. $K(\nu,A)=\SSYT(\nu,A)$ \. the number of semistandard
Young tableaux of shape $\nu$ and weight \ts
$(a_{11},a_{12},\ldots, a_{\ell m})$.

\medskip

\begin{prop}\label{p:Kron-Kostka-CT}
Let $\la,\mu,\nu\vdash n$ and suppose $A \unlhd \CT(\la,\mu)$.  Then:
$$
g(\la,\mu,\nu) \, \le \, \sum_{B\in \CT(\la,\mu)} \. K(\nu,B)
\, \le \, \rT(\la,\mu) \. \cdot \. K(\nu,A) \, \le \, \rT(\la,\mu) \. \cdot \.
\rT(\nu,A)\ts.
$$
\end{prop}

\begin{proof}[First proof]
For the first inequality, recall a result by James and Kerber
\cite[Lemma~2.9.16]{JK} that
\begin{equation}\label{eq:JK-lemma}
\phi^\la \cdot \phi^\mu \, = \, \sum_{B\in \CT(\la,\mu)} \. \phi^B
\, = \, \sum_{B\in \CT(\la,\mu)} \. \sum_{\nu\rdom B} \. \chi^\nu\.,
\end{equation}
where $\phi^B$ is an induced representation corresponding to
ordering of \ts $\{b_{ij}\}$ in the contingency table \ts
$B=(b_{ij})\in \CT(\la,\mu)$.  On the other hand,
\begin{equation}\label{eq:JK-dom}
\phi^\la \cdot \phi^\mu \, = \, \sum_{\al\rdom\la} \. \sum_{\be\rdom\mu} \.
K(\al,\la) \ts K(\be,\mu)  \, \bigl[\chi^\al\cdot \chi^\be\bigr] \, \ge \,
\chi^\la \cdot \chi^\mu \, = \,
\sum_{\nu\vdash n} \. g(\la,\mu,\nu) \. \chi^\nu\.,
\end{equation}
where the inequality follows from Theorem~\ref{t:Kostka-dom} by taking
the term $\al=\la$ and $\be=\mu$. Comparing terms in $\chi^\nu$ in
equation~\eqref{eq:JK-lemma} and inequality~\eqref{eq:JK-dom} finishes
the proof of the first inequality.

The second inequality follows immediately from Theorem~\ref{t:Kostka-dom} and
the definition of~$A$. The last inequality follows from
Lemma~\ref{l:Kostka-CT} above.
\end{proof}

\begin{proof}[Second proof]
Another way to see all the inequalities in the statement is through the following Schur function
identities and inequalities:
\begin{align*}
g(\la,\mu,\nu) \, &= \, \< \, s_{\nu}[\. \bx \cdot \by \ts] \, , \, s_\la(\bx) \. s_\mu(\by)\, \>  \ \, = \, \
\Bigl\< \, \sum_{B} \. K(\nu,B) \. \prod_{i,j} \. (x_i y_j)^{B_{i,j}} \, , \, s_\la(\bx)\. s_\mu(\by) \,\Bigr\>  \\
&\leq \,
\sum_B \. K(\nu,B) \. \Bigl\< \, \bx^{row B} \. \by^{col B} \, , \, \Bigl( \, \sum_{\al \unlhd \la} \. K(\la,\al) \. s_\al(\bx)\, \Bigr) \Bigl(\, \sum_{\be \unlhd \mu} \. K_{\mu,\be} \. s_{\be}(\by) \, \Bigr) \,\Bigr\rangle \\
&\le \, \sum_B \. K(\nu, B) \. \bigl\< \. \bx^{row B} \. \by^{col B} \, , \, h_\la(\bx)\. h_\mu(\by) \. \bigr\> \, \  = \, \ \sum_{B \in \CT(\la,\mu)} \.  K(\nu,B) \, ,
\end{align*}
since the monomial  and the homogenous symmetric functions are orthonormal.
Here $B$ goes through all $2$-dimensional tables and $rowB, colB$ are the row/column sums of its entries.
\end{proof}

\smallskip

\begin{ex}{\rm
Let $n=\ell^3$, $k=\ell^2$, $\la=\mu=(k^\ell)$, $\al=\ell$, and $\ell\to \infty$.
Then
$$
\rT(\la,\mu) \, \le \, \exp \bigl[\ell^2\log \ell \. + \. O(\ell^2)\bigr]\..
$$
by Example~\ref{ex:motiv-2dim}.  We can take $A = (a_{ij})$, where
$a_{ij}=\ell$ for all $1\le i,j\le \ell$.  Denote
by \ts $\ga = \bigl(\ell^{\ell^2}\bigr)$ \ts the corresponding partition.
Let $\nu\vdash n$, s.t.\ $\ell(\nu)=r$.   Then
$$\aligned
g(\la,\mu,\nu) \. & \le \. \rT(\la,\mu)\cdot K(\nu,\ga) \. \le \.
\exp \bigl[\ell^2\log \ell \ts + \ts O(\ell^2)\bigr] \. \cdot \.
\binom{\ell+r-1}{r-1}^{\ell^2}
 \\
 & \le \, \exp \bigl[\ell^2\log \ell \ts + \ts O(\ell^2)\bigr] \. \cdot \.
 (\ell+r-1)^{(r-1)\ell^2} \, \le \, \exp\.\bigl[r\ts \ell^2\ts \log(\ell+r) \. + \. O(\ell^2)\bigr]
\endaligned
$$
Here the second inequality follows from the structure of $\SSYT(\nu,\ga)$.
We need to place $\ell$ numbers~$a$ into $r$~rows, for all $1\le a\le \ell^2$.
Note that these favorably compare to the dimension bounds for $r=o(\ell)$,
see Example~\ref{ex:motiv-HLF}.
}\label{ex:motiv-Kostka}
\end{ex}

\begin{rem}{\rm The reason these estimates are reasonable, is because
the irreducible rep $\SS^\la$ is the largest part of the induced representation
\ts $\MM^\la$. On the other hand, even in this case the lower
bound is not expected to be anywhere close.  Cf.\ \cite[Cor.~3.12]{PPY},
which gives $\la=(a-1)^{a^2}$ and $g(\la,\la,\la)=0$.
}
\end{rem}

\begin{thm}\label{t:Kron-Kostka-rows}
Let $\la,\mu,\nu\vdash n$ such that $\ell(\la)=\ell$, $\ell(\mu) = m$,
and $\ell(\nu)=r$.  Then:
$$g(\la,\mu,\nu) \, \le \, \left(1+\frac{\ell m r}{n}\right)^n \left(1+\frac{\ell m}{n}\right)^n
\left(1+\frac{n}{\ell m r}\right)^{\ell m r} \left(1+\frac{n}{\ell m}\right)^{\ell m}\ts.
$$
\end{thm}

\medskip

\begin{rem}{\rm When $\la=\mu'\vdash n$, the smallest $A^\ast$ is a $0/1$ matrix,
so $\ga=(1^n)$. In this case $K(\nu,\ga) = f^\nu$ and the bound in the theorem is
useless since we already have $g(\la,\mu,\nu)\le f^\nu$.  This is why it helps to
have $\ell(\la), \ell(\mu) = o(\sqrt{n})$.
}
\end{rem}

\bigskip

\subsection{Proof of the upper bound} We start with the following useful result:

\begin{lemma} \label{l:G-maj}
Let \ts $\la=(\la_1,\ldots,\la_\ell)$, $\al = (\al_1,\ldots,\al_\ell)$,
$\la,\al \in \rr_+^\ell$.  Similarly, let \ts $\mu = (\mu_1,\ldots,\mu_m)$,
$\be = (\be_1,\ldots,\be_m)$, $\mu,\be\in \rr_+^m$.  Denote by
$G(\la,\mu):=e^{g(Z)}$ the upper bound in Theorem~\ref{t:barv-2dim}.
Then:
$$G(\la,\mu) \, \le \, G(\al,\be) \quad \text{for all}
\quad \la \rdom \al, \, \mu \rdom \be\ts.
$$
\end{lemma}

\begin{proof}
As in the discrete case, it is easy to see that if  two weakly decreasing sequences $\la,\al$ majorize one another, $\la \rdom \al$, then $\la$ can be obtained from $\al$ by a finite sequence of operations adding vectors of the form $\eps e(i,j)$, where $e(i,j)_r =0$ for $r \neq i,j$ and $e(i,j)_i=1, e(i,j)_j=-1$, see e.g.~\cite[$\S$2]{MOA}. Now it is enough to prove the inequality in the case when $\mu=\be$ and $\la = \al + \eps e(i,j)$, and apply it consecutively in the algorithm obtaining $(\al,\be)$ from $(\la,\mu)$ by changing $\al$ to $\la$ first, and then $\be$ to $\mu$.

Let $w \in \mathbb{R}_+^{\ell^2}$ be the unique maximizer of $g(Z)$ for $G(\la,\mu)$, and let $\al = \la-\eps e(i,j)$ and assume for simplicity that $i=1$, $j=2$. Consider the $2\times \ell$ section with rows $1,2$, and let its column margins be $a_1,\ldots,a_\ell$. We have:
$$\sum_{i=1}^\ell \ts  \bigl(w_{1i} - w_{2i}\bigr) \. = \. \la_1-\la_2 \. \geq \. 2\ts\eps\ts,
$$
so the positive terms among \ts $(w_{1i}-w_{2i})$ \ts add up to at least $2\eps$.  Assume for simplicity that the positive terms are for \ts $i=1,\ldots,r$, and choose \ts $0 \leq \eps_i \leq \frac12 (w_{1i}-w_{2i})$, so that $\eps_1+\ldots +\eps_r=\eps$. Let \ts
$z_{ij} = w_{ij}$ \ts for $i\neq 1,2$ or $j>r$, and let \ts $z_{1j}=w_{1j} - \eps_j$, $z_{2j} = w_{2j} +\eps_j$. Then $z$ has margins $(\al,\mu)$, and we will show that $g(z) \geq g(w)$.

To see this, let $f(x) = (1+x)\log(1+x) - x \log x$, and note that
$f(a-x)+f(b+x)$ is increasing when $a > b$ and $x\in[0, \frac{a-b}{2}]$. Hence,
$$f(z_{1j})+f(z_{2j}) \. = \. f(w_{1j} - \eps_j) + f(w_{2j}+\eps_j) \. \geq \. f(w_{1j})+f(w_{2j})$$
for $j=1,\ldots,r$, and equal for the other indices.  Thus, $g(z) \geq g(w)$.
We have:
$$G(\al,\mu) \. = \. \max_{Z \in \mathcal{P}(\al,\mu) } \exp g(Z) \. \geq \.
\exp g(z) \. \geq \.\exp g(w) \. = \. G(\la,\mu)\ts,
$$ which completes the proof.
\end{proof}

\bigskip

Note that $\al,\be$ in the lemma are not necessarily integral.
When they are in fact integer partitions, this lemma on majorization
over margins is motivated by the similar majorization
inequality for the number of contingency tables:
$$
\rT(\la,\mu) \ \le \rT(\al,\be) \quad \text{for all} \quad
\la \rdom \al, \, \mu \rdom \be\ts,
$$
see e.g.~\cite{B1}.

\begin{proof}[Proof of Theorem~\ref{t:Kron-Kostka-rows}]
Let $\tau = \bigl(\lceil n/\ell m\rceil^a\lfloor n/\ell m\rfloor^b \bigr)$, s.t.\
$a+b=\ell m$. Clearly,  $\tau \dom \CT(\la,\mu)$.
By Proposition~\ref{p:Kron-Kostka-CT},  we have:
$$
g(\la,\mu,\nu) \, \le \, \rT(\la,\mu) \. \cdot \.
\rT(\nu,\tau)\ts.
$$
Let $\al:=(n/\ell)^\ell$, $\be := (n/m)^m$, $\ga := (n/r)^r$, $\om=(n/\ell m)^{\ell m}$.
By Theorem~\ref{t:barv-2dim} and Lemma~\ref{l:G-maj}, we have:
$$\aligned
\rT(\la,\mu)  \, & \le \,  G(\la,\mu) \, \le \, G(\al,\be) \\
\rT(\nu,\tau)  \, & \le \,  G(\nu,\tau) \, \le \, G(\ga,\om)\ts.
\endaligned
$$
A direct calculation gives:
$$
\aligned
G(\al,\be) \, & = \, \exp \. \left[\ell\ts m \cdot
\left[\left(\frac{n}{\ell m} + 1\right) \log \left(\frac{n}{\ell m} + 1\right) \. - \. \frac{n}{\ell m} \log \frac{n}{\ell m}\right]\right] \\
& = \, \exp \.  \left[n \.\log \left(\frac{\ell m}{n} + 1\right) \. + \. m\ts\ell \.\log \left(\frac{n}{\ell m} + 1\right)\right]\,,
\endaligned
$$
and
$$
\aligned
G(\ga,\om) \, & = \, \exp \. \left[\ell\ts m\ts r \cdot
\left[\left(\frac{n}{\ell m r} + 1\right) \log \left(\frac{n}{\ell m r} + 1\right) \. - \. \frac{n}{\ell mr} \log \frac{n}{\ell m r}\right]\right] \\
& = \, \exp \.  \left[n \.\log \left(\frac{\ell m r}{n} + 1\right) \. + \. \ell\ts m\ts r \.\log \left(\frac{n}{\ell m r} + 1\right)\right]\.
\endaligned
$$
Combining the bound above gives the result.
\end{proof}


\bigskip

\section{The 3-dimensional CTs approach} \label{sec:3CT}

Let $\la,\mu,\nu\vdash n$.  Denote by $\rT(\la,\mu,\nu)$
the number of $3$-dimensional contingency tables with $2$-dimensional
sums given by $\la$, $\mu$ and $\nu$.

\begin{thm} \label{t:Kron-3dim-CT}
We have: \.
$g(\la,\mu,\nu) \ts \le \ts\rT(\la,\mu,\nu)$.
\end{thm}

\begin{proof}  Recall \emph{Schur's theorem}\footnote{Sometimes
also called \emph{generalized Cauchy identity}.}~\cite[Exc.~7.78f]{EC2}, that:
$$
\sum_{(\la,\mu,\nu)\in \cP^3} \. g(\la,\mu,\nu) \, s_\la(\bx) \. s_\mu(\by) \. s_\nu(\bz) \,
= \, \sum_{(\al,\be,\ga)\in \cP^3} \. \rT(\al,\be,\ga) \, m_\al(\bx) \. m_\be(\by) \ts m_\ga(\bz) \,.
$$
Taking the coefficients in \ts $\bx^\al \ts \by^\be \ts \bz^\ga$ \ts on both sides gives:
$$
\rT(\al,\be,\ga) \, = \, \sum_{\la \rdom \al, \ts \mu\rdom \be, \ts \nu\rdom \ga} \.
g(\la,\mu,\nu) \. K_{\la\al} \. K_{\mu\be} \. K_{\nu\ga} \, \ge \,g(\al,\be,\ga)\ts,
$$
where the inequality follows from \ts $K_{\al\al}=1$ \ts for all $\ts \al\vdash n$.
\end{proof}

\smallskip

\begin{ex}{\rm Let $n=\ell^3$, $k=\ell^2$,
$\la=\mu=(k^\ell)$, and $\ell\to \infty$.
Let $\nu = (n/r)^r \vdash n$, where \ts $\ell(\nu)=r=o(\ell)$.
By Theorem~\ref{t:barv-3dim}, we have:
$$g(\la,\mu,\nu) \. \le \. \rT(\la,\mu,\nu) \. \le \.
G(\la,\mu,\nu) \. = \.  \exp \ts g(Z),
$$
where $Z=(z_{ijs})$, $z_{ijs}=1/\ell^2r$.  We have:
$$
\aligned
g(Z) \, & = \, \ell^2r \ts \left[ \left(\frac{n}{\ell^2 r} +1\right)\log\left(1+\frac{n}{\ell^2 r}\right) \. - \.
\left(\frac{n}{\ell^2 r}\right)\log\left(\frac{n}{\ell^2 r}\right)\right] \\
& = \,
n \ts \log\left( 1 + \ell^2r/n\right) \. + \. \ell^2r \ts \Bigl[\log n \.
+ \. \log\left(1+\ell^2r/n\right) \. - \. \log(\ell^2 r) \Bigr]\\
& = \, \ell^2r + O\left((\ell^2 r)^2/n\right) \.  + \. \ell^2 r\ts \left(3\log\ell \. + \. O(\ell^2 r/n) \. - \.
2\log\ell-\log r \right) \\
& = \, \ell^2r \ts \log \ell \. - \.\ell^2r \ts \log r \. + \. O(\ell^2 r)\ts.
\endaligned
$$
Therefore,
$$
g(\la,\mu,\nu) \, \leq \, \exp \. \bigl[\ell^2r \ts \log \ell \. - \.\ell^2r \ts \log r \. + \. O(\ell^2 r) \bigr]\ts.
$$
Note that this is a stronger bound than the one in Example~\ref{ex:motiv-Kostka}.
}\label{ex:motiv-3dim}
\end{ex}

\smallskip

\begin{thm}\label{t:Kron-3dim-rows}
Let $\la,\mu,\nu\vdash n$ such that $\ell(\la)=\ell$, $\ell(\mu) = m$,
and $\ell(\nu)=r$.  Then:
$$g(\la,\mu,\nu) \, \le \, \left(1+\frac{\ell m r}{n}\right)^n
\left(1+\frac{n}{\ell m r}\right)^{\ell m r} .
$$
\end{thm}

\smallskip

This is clearly a better bound than the one in Theorem~\ref{t:Kron-Kostka-rows}.

\begin{proof}
In notation the proof of
Theorem~\ref{t:Kron-Kostka-rows}, we have \ts $G(\al,\be,\ga)=G(\ga,\om)$.
On the other hand, \ts $g(\la,\mu,\nu) \le G(\al,\be,\ga)$ \ts
by Theorem~\ref{t:Kron-3dim-CT}.  The result follows.
\end{proof}

\medskip

\begin{cor}\label{c:Kron-3dim-rows-asy}
Let $\la,\mu,\nu\vdash n$ such that $\ell(\la)=\ell$, $\ell(\mu) = m$,
and $\ell(\nu)=r$.    Suppose $\ts \ell \ts m \ts r \ts = o(n)$ \ts as \ts $n \to \infty$.
Then:
$$g(\la,\mu,\nu) \, \le \, \exp \. \bigl[\ell \ts m \ts r \ts \log n \.
+ \. O(\ell \ts m \ts r)\bigr]\..
$$
\end{cor}

\medskip

\begin{rem}{\rm
In summary, we obtain stronger general bounds in
Theorem~\ref{t:Kron-3dim-rows} using contingency tables than that in
Theorem~\ref{t:Kron-Kostka-rows} using Kostka numbers. At the same
time, more careful estimates on Kostka numbers for $r=O(1)$ in
Example~\ref{ex:motiv-Kostka} give the same bounds than that in
the Example~\ref{ex:motiv-3dim} via contingency tables.  It is
unclear if upper bounds on Kostka numbers can be improved to
beat bounds on contingency tables in full generality.
}\end{rem}

\bigskip

\section{Bounds for the reduced Kronecker coefficients} \label{sec:Reduced}

The \emph{reduced Kronecker coefficients} were introduced by
Murnaghan in 1938 as the stable limit of \emph{Kronecker coefficients},
when a long first row is added:
\begin{equation}\label{eq:red-def}
\rg(\al,\be,\ga) \. := \. \lim_{n\to \infty} \. g\bigl(\al[n],\be[n],\ga[n]\bigr),
\end{equation}
where 
$$
\al[n]:= (n-|\al|,\al_1,\al_2,\ldots), \quad \text{for all} \quad n\ge |\al|+\al_1\ts,
$$
see~\cite{M1,M2}.
They generalize the classical \emph{Littlewood--Richardson} (LR--) \emph{coefficients}:
$$
\rg(\al,\be,\ga) \, = \, c^\al_{\be\ga} \quad \text{for} \quad |\al|\. = \. |\be| \ts +\ts |\ga|\ts,
$$
see~\cite{Lit}.  As such, they occupy the middle ground between the Kronecker
and the LR--coefficients.  We apply the results above to give an upper bound for
the reduced Kronecker coefficients with few rows. Note that only other upper bound
we know for $\rg(\al,\be,\ga)$ is the upper bound for the maximal value given
in~\cite{PP3}.

\smallskip

\begin{thm}\label{t:Kron-red-rows}
Let $\al\vdash a$, $\be\vdash b$, $\ga\vdash c$, such that $\ell(\al)=\ell$,
$\ell(\be) = m$, and $\ell(\ga)=r$. Denote $N:=a+b+c$.  Then:
$$\rg(\al,\be,\ga) \, \le \, \sum_{n=0, \, n=N\.\text{\rm mod\. 2}}^{\min\{a,b,c\}} \,
\rE(\ell m r,\ts n) \cdot \rE\left(\ell m, v-c\right)\cdot
\rE\bigl(\ell r, v-b)\cdot \rE(m r, v-a)\ts,
$$
where \. 
$$\rE(s,w) \, := \, \left(1+\frac{s}{w}\right)^w
\left(1+\frac{w}{s}\right)^{s}   \quad \text{for} \ \ \. w \. > \. 0\ts,
$$
$$\rE(s,0) \, := \, 1\ts,  
\quad \text{and} \quad
v\. := \. \frac{1}{2} \bigl(N-3n\bigr)\ts.
$$
\end{thm}

\smallskip

\nin
The proof follows the previous pattern.  We begin with
the following combinatorial lemma.

\begin{lemma} \label{l:CT-reduced}
$$g(\al,\be,\ga) \. \le \. \rR(\al,\be,\ga)\ts,
$$
where
$$
\rR(\al,\be,\ga) \, = \, \sum_{\Phi(\al,\be,\ga)} \, \rT(\la,\mu,\nu) \,
\rT(\pi,\rho)\, \rT(\si,\tau) \, \rT(\eta,\zeta)\ts,
$$
and
$$\Phi(\al,\be,\ga) \ts := \ts \bigl\{(\la,\mu,\nu,\pi,\rho,\si,\tau,\eta,\zeta)\in\cP^9,
(\al,\be,\ga) = (\la,\mu,\nu)+(\pi,\rho,\emp)+(\si,\emp,\tau) +(\emp,\eta,\zeta)
\bigr\}.
$$
\end{lemma}

\begin{proof}
The result follows from the identity given recently in~\cite{BR}, and by taking the coefficients
in \ts $\bx^\al\. \by^\be\. \bz^\ga$ \ts on both sides:
$$\aligned
& \sum_{(\al,\be,\ga)\in \cP^3} \. \rg(\al,\be,\ga) \, s_\al(\bx) \. s_\be(\by) \. s_\ga(\bz) \\
& \qquad = \, \left[\prod_{i,j,k=1}^\infty \. \frac{1}{1-x_iy_jz_k}\right] \, \left[\prod_{i,j=1}^\infty \. \frac{1}{1-x_iy_j} \right]
\, \left[\prod_{i,k=1}^\infty \. \frac{1}{1-x_iz_k}\right] \, \left[\prod_{j,k=1}^\infty \. \frac{1}{1-y_jz_k}\right] \\
& \qquad = \, \sum_{(\al,\be,\ga)\in \cP^3} \. \rR(\al,\be,\ga) \, m_\al(\bx) \. m_\be(\by) \ts m_\ga(\bz) \,.
\endaligned
$$
Here the nine partitions in the definition of \ts $\Phi(\al,\be,\ga)$ \ts
come from a combinatorial interpretation of one $3$-dimensional and three $2$-dimensional
contingency tables.  \end{proof}

\medskip

\begin{proof}[Proof of Theorem~\ref{t:Kron-red-rows}]
In notation of Lemma~\ref{l:CT-reduced}, denote \ts $n:=|\la|=|\mu|=|\nu|$.
Note that \ts $n \le a,b,c,$ \ts and has the same parity as $N=a+b+c$.
Observe that $n$ determines the remaining six partitions sizes: \ts
$\pi,\rho \vdash v-c$, \ts $\si,\tau\vdash v-b$, and
\ts $\eta,\zeta\vdash v-a$.
By assumption as in the theorem,
we have an upper bound on the number of rows of all nine partitions:
\ts $\ell(\la),\ell(\pi),\ell(\si)\le \ell$, \ts
$\ell(\mu), \ell(\rho),\ell(\eta)\le m$,  and
\ts $\ell(\nu),\ell(\tau), \ell(\zeta)\le r$.  In notation of the theorem,
from the proof in the previous section we obtain:
$$\rT(\la,\mu,\nu) \. \le \. G\bigl((n/\ell)^\ell,(n/m)^m,(n/r)^r\bigr)
\. = \. \rE(\ell m r,\ts n)\ts.
$$
The same holds for the remaining pairs of partitions:
$$
\aligned
\rT(\pi,\rho) \. & \le \. G\bigl(((v-c)/\ell)^\ell,((v-c)/m)^m\bigr)
\. = \. \rE(\ell m ,\ts v-c)\ts,\\
\rT(\si,\tau) \. & \le \. G\bigl(((v-b)/\ell)^\ell,((v-b)/r)^r\bigr)
\. = \. \rE(\ell r ,\ts v-c)\ts,\\
\rT(\eta,\zeta) \. & \le \. G\bigl(((v-a)/m)^m,((v-a)/r)^r\bigr)
\. = \. \rE(m r,\ts v-c)\ts.
\endaligned
$$
Note that in the case of empty partitions of zero, we have a unique
zero contingency table, giving $\rT(\cdot) = \rE(\cdot,0)=1$, by definition 
in the statement of the theorem. 
Putting all these bounds into the main inequality in
Lemma~\ref{l:CT-reduced}, implies the result.
\end{proof}

\smallskip

\begin{ex} {\rm
Suppose $a=N/6$, $b=N/3$ and $c=N/2$, with
$\ell=m=r=N^{1/3}$. Then the sum over $n$ in Theorem~\ref{t:Kron-red-rows}
maximizes at $n=a=N/6$. Thus $v=N/2$, and
we have:
$$\aligned
\rg(\al,\be,\ga) \, &\le \, (n/2) \cdot \rE(N,N/6) \cdot \rE\bigl(N^{2/3},0\bigr)
\cdot  \rE\bigl(N^{2/3},N/6\bigr)\cdot \rE\bigl(N^{2/3},N/3\bigr) \\
& \le  \,
\left(1+\frac{1}{6}\right)^{N}  \left(1+6\right)^{N/6} \,\cdot\, \exp O\bigl(N^{2/3}\bigr) \\
& \le  \, \exp\ts \bigl[c\ts N \. + \. O\bigl(N^{2/3}\bigr)\bigr].
\endaligned
$$
where \ts $c=(\log 7/6) + (\log 7)/6 \approx 0.4785$.

Note that in the limit~$(\circ)$ it suffices to take \ts
$n  \ts \ge \ts |\al| \ts + \ts |\be| \ts + \ts |\ga|$,
see~\cite{BOR,Val}.  Since the number of rows increases only by~1
under the limit, Theorem~\ref{t:Kron-3dim-rows} still gives
an exponential bound for \ts $\rg(\al,\be,\ga)$, but with a
much larger constant.
}
\end{ex}

\bigskip

\section{Multi-LR approach} \label{sec:multi-LR}

Let $K_{\la,\mu}^{\<-1\>}$ be the inverse of the Kostka matrix.
Recall that $K_{\la,\mu}^{\<-1\>}=0$ unless $\la \rdom \mu$.

\begin{lemma} For all $\la,\mu\vdash n$, we have:
$$\bigl|K_{\la,\mu}^{\<-1\>}\bigr| \. \le \. \ell(\mu)!
$$
\end{lemma}

\begin{proof}[First proof]  Let $\ell=\ell(\la)$, $m=\ell(\mu)$.
By~\cite[Theorem~1]{ER}, $K_{\la,\mu}^{\<-1\>}$ is a signed
sum of the ``special'' ribbon (rim hook) tableaux of shape $\mu$ with hook length~$\la_i$.
By the idefinition in~\cite{ER} and trivial induction, these are uniquely determined
by positions of starting squares of ribbons along the first column.  Thus
the number of ribbon tableaux is at most $m(m-1)\cdots (m-\ell+1)\le m!$, as desired.
\end{proof}

In special cases, the combinatorial interpretation in~\cite{ER}
can perhaps lead to further improvements of the upper bound.
For completeness, we enclose a short proof which avoids~\cite{ER}.

\begin{proof}[Second proof]
Note that these are the coefficients in the Schur function expansion \ts
$$s_\la \. = \, \sum_\mu \. K_{\la,\mu}^{\<-1\>} \ts h_\mu\.,
$$
see e.g.~\cite{Mac,EC2}.
On the other hand, by the Jacobi--Trudi identity we have:
$$
s_\la \. = \. \det \Bigl[ h_{\la_i -i +j }\Bigr]_{i,j=1}^{\ell(\la)}\ts.
$$
This gives:
$$\sum_{\mu\vdash n} \. \bigl|K_{\la,\mu}^{\<-1\>}\bigr| \. = \. \ell(\la)!
$$
By Theorem~\ref{t:Kostka-dom}, we have \ts $\ell(\la)\leq \ell(\mu)$,
and the inequality follows.
\end{proof}

\begin{thm}[{\cite{Val}}] Let $\la,\mu,\nu\vdash n$ and $\ell(\nu)=r$.
$$
g(\la,\mu,\nu) \, = \, \sum_{\pi\rdom \nu} \. K_{\pi,\nu}^{\<-1\>} \, \LR(\la,\mu\ts|\ts\pi),
$$
where
$$
\LR(\la,\mu\ts|\ts\pi) \, = \, \sum_{\rho^{(1)}\vdash \pi_1,\ldots,\rho^{(s)}\vdash \pi_s}
c\bigl(\la \ts|\ts \rho^{(1)},\ldots,\rho^{(s)}\bigr)\cdot c\bigl(\mu \ts|\ts \rho^{(1)},\ldots,\rho^{(s)}\bigr),
$$
$s=\ell(\pi)\le r$, and $c\bigl(\la \ts|\ts \rho^{(1)},\ldots,\rho^{(s)}\bigr)$ is a multi-LR coefficient.
\end{thm}

\begin{thm}[{\cite[Cor.~4.12]{PPY}}]
Let $\la\vdash n$, $|\mu|+|\nu|=n$, $\ell=\ell(\la)$. Then
$$c^{\la}_{\mu\ts\nu} \. \le \. (\la_1+\ell)^{\ell^2/2}
$$
\end{thm}

\begin{lemma}  In the notation above, let $\ell(\la)=\ell$, $\ell(\pi)=s$. Then:
$$
c\bigl(\la \ts|\ts \rho^{(1)},\ldots,\rho^{(s)}\bigr) \, \le \, p(n)^{s-1} \cdot (\ell+\la_1)^{s\ts \ell^2/2}
$$
\end{lemma}

\begin{proof}
This follows from the definition of multi-LR coefficient in~\cite{Val} as an increasing
chain of LR-tableaux of length~$s$.  The number of such chains is trivially bounded
by $p(n)^{s-1}$,  while the numbers of each LR-tableaux is at most
$(\la_1+\ell)^{\ell^2/2}$ from above.
\end{proof}

\begin{thm} \label{t:Kron-multi}
Let $\ell(\la)=\ell$, $\ell(\mu) = m$, $\ell(\nu)= r$. Then:
$$
g(\la,\mu,\nu) \, \le \, r! \ts \cdot \ts p(n)^{3r-2} \ts \cdot\ts n^{r-1} \ts \cdot\ts
(\ell+\la_1)^{r\ts \ell^2/2} (m+\mu_1)^{r\ts m^2/2}
$$
\end{thm}

\begin{proof}
This bound follows by combining the lemma and estimates above.  Indeed, in the
summation for $\LR(\cdot)$, the sum is over at most $p(n)^r$ terms.
On the other hand, there are at most at most $n^{r-1}$ partitions~$\pi$
in the summation for $g(\la,\mu,\nu)$. This follows from $\pi \rdom \nu$,
and the fact that there are at most $n^{r-1}$ partitions~$\pi$ with $\ell(\pi)\le r$.
\end{proof}

\medskip

\begin{cor}\label{c:Kron-multi-LR-asy}
Let $\la,\mu,\nu\vdash n$ such that $\ell(\la)=\ell$, $\ell(\mu) = m$,
and $\ell(\nu)=r$.    Suppose $\ts \ell^2r = o(n)$, \ts $m^2 r = o(n)$ \ts as \ts $n \to \infty$.
Then:
$$\aligned
g(\la,\mu,\nu) \, & \le \, \exp \. \left[\frac12\. \ell^2 \ts r \ts \log (\ell+\la_1) \.
+\. \frac12\.  m^2\ts r \ts \log (m+\mu_1)  \.
+ \. O\bigl(r\ts\sqrt{n}\bigr)\right] \\
\, & \le \, \exp \. \left[\frac{1}{2}\bigl(\ell^2 \ts +\ts m^2\bigr) \ts r \ts \log n \.
+ \. O\bigl(r\ts\sqrt{n}\bigr)\right].
\endaligned
$$
\end{cor}

\begin{rem}{\rm
This bound is \emph{always} asymptotically weaker than that in
Corollary~\ref{c:Kron-3dim-rows-asy} for $\ell \ne m$.  On the other
hand, for \ts $\ell,m=o\bigl(n^{1/4}\bigr)$ \ts the dominant term is
\ts $O\bigl(r\ts\sqrt{n}\bigr)$, which is too crude a
bound for the number of diagram chains.
}\end{rem}

\bigskip

\begin{ex}{\rm
As before, let $n=\ell^3$, $k=\ell^2$,
$\la=\mu=(k^\ell)$, and $\ell\to \infty$.
Let $\nu = (n/r)^r \vdash n$, where \ts $\ell(\nu)=r\le \ell$.
The theorem gives:
$$
g(\la,\mu,\nu)
\, \le \, \exp \ts \bigl[r\ell^2\log \ell \. + \. O(r\ell^{3/2})\bigr] \ts.
$$
This bound gives the same leading term than that in examples~\ref{ex:motiv-Kostka}
and~\ref{ex:motiv-3dim} and has the error term somewhere in between.
}\label{ex:motiv-multi-LR-Kron}
\end{ex}

\bigskip

\begin{rem}{\rm
Note that the bounds in Theorem~\ref{t:Kron-multi} are tight for \ts
$\ell(\la), \ell(\mu)=\Theta(\sqrt{n})$, since \emph{all terms}
in the product become $\exp\Theta(n\log n)$, while we already
know that for \emph{some} such $\la,\mu,\nu$ we get $g(\la,\mu,\nu)$
of the same order:
$$
\max_{\la,\mu,\nu\vdash n} \. g(\la,\mu,\nu) \, \ge \, \sqrt{n!} \. \exp \bigl[-c\ts\sqrt{n}\bigr]
\quad \text{for some} \ \. c>0,
$$
see~\cite[$\S$3]{PPY}.
}\end{rem}

\bigskip

\section{Binary contingency tables approach} \label{sec:3CT-bin}

\subsection{Upper bound}
Let $\la,\mu,\nu\vdash n$.  As before, denote by $\BCT(\la,\mu,\nu)$ and
$\rB(\la,\mu,\nu)$ the set and the number of $3$-dimensional \emph{binary}
contingency tables, respectively.

\begin{thm}[{see~$\S$\ref{ss:fin-rem-history}}]  \label{t:Kron-3dim-CT-binary}
We have: \. $g(\la,\mu,\nu) \ts \le \ts\rB(\la',\mu',\nu')$.
\end{thm}

Although this result is well known in the literature,
we include two short proofs for completeness.

\begin{proof}[First proof]
By analogy with the proof of Theorem~\ref{t:Kron-3dim-CT}, recall
the \emph{dual Schur's theorem}:
$$
\sum_{(\la,\mu,\nu)\in \cP^3} \. g(\la,\mu,\nu) \, s_{\la'}(\bx) \. s_{\mu'}(\by) \. s_{\nu'}(\bz) \,
= \, \sum_{(\al,\be,\ga)\in \cP^3} \. \rB(\al,\be,\ga) \, m_\al(\bx) \. m_\be(\by) \ts m_\ga(\bz) \,.
$$
The result now follows by taking the coefficients in \ts $\bx^\al \ts \by^\be \ts \bz^\ga$
\ts on both sides.
\end{proof}

\begin{proof}[Second proof] Recall the proof of the first
inequality in Proposition~\ref{p:Kron-Kostka-CT}.
Equation~\eqref{eq:JK-lemma} states:
$$
\phi^\la \cdot \phi^\mu \,\, = \, \sum_{B\in \CT(\la,\mu)} \. \phi^B\..
$$
Now, for every \ts $B=(b_{ij})\in \CT(\la,\mu)$ \ts take \ts
$X_B:=(x_{ijk})\in \BCT(\la,\mu,B')$, where $x_{ijk} = 1$ if
$k \le b_{ij}$ and $x_{ijk} = 0$ otherwise.  Taking the binary
tables with $B'=\nu'$, we conclude:
$$
g(\la,\mu,\nu) \, \le \, \sum_{B\in \CT(\la,\mu)} \. K(\nu,B)
\, \le \, \rB(\la,\mu,\nu')\.,
$$
and the result follows from the symmetry \ts $g(\la,\mu,\nu') = g(\la',\mu',\nu')$.
\end{proof}

\bigskip

\begin{ex}{\rm Let $\la=\mu=\nu=(m+1,1^m)$, $n=2m+1$.  Note that in this
case \ts $\rB(\la,\mu,\nu) > m!$, by a placing $(m+1)$ ones along
a line and a permutation in $S_m$ into an orthogonal $2$-plane
On the other hand, \ts $g(\la,\mu,\nu)\le
f^\la=\binom{2m}{m}< 4^m$ \ts is a much
better estimate.  In fact, $g(\la,\mu,\nu)=1$ in this case,
see e.g.~\cite{Rem,Ros}.
}\end{ex}

\begin{ex}{\rm Let $\la=\mu=\nu=(\ell^k)$, where $n=\ell^3$, $k=\ell^2$.
Then \ts $\rB(\la,\mu,\nu)=1$ since there is a unique \ts
$\ell\times\ell\times\ell$ \ts binary table with all 2-dim sums $\ell^2$.
It is easy to see directly that $g(\la,\mu,\nu)=1$ in this case.
}\label{ex:CT-binary-rect-upper-1}\end{ex}

\begin{rem}{\rm
Note that Theorem~\ref{t:Kron-3dim-CT-binary} and the upper bound
in Theorem~\ref{t:barv-3dim-binary} give \emph{some} upper bound
on the Kronecker coefficients $g(\la,\mu,\nu)$.  We do not include
this bound since majorization property does not hold for the number
$\rB(\la,\mu,\nu)$ of binary contingency tables, and as a result
there is no closed form inequality in terms of the numbers of rows
of three partitions.
}\end{rem}

\medskip

\subsection{Tensor squares}
Let $\la=\nu\vdash n$, $\ell(\la)=\ell$, $\la_1=m$, $\ell(\nu)=r$.
By Theorem~\ref{t:Kron-3dim-CT-binary}, we have:
$$
g(\la,\la,\nu) \. = \. g(\la',\la,\nu') \. \le \. \rB(\la,\la',\nu)\ts.
$$
Let us show that this bound is weaker than the dimension bound:

\begin{prop} \label{p:CT-bin-dim-Kron-square}
We have: \ts $f^\nu \ts \le \ts  \rB(\la,\la',\nu)$.
\end{prop}

This implies that the upper bounds in the theorem cannot disprove
the Saxl conjecture.  In fact, in sharp contrast with
examples~\ref{ex:motiv-Kostka} and~\ref{ex:motiv-3dim},
Theorem~\ref{t:Kron-3dim-CT-binary} does not give
non-trivial upper bound for any tensor square.

\begin{proof}[Proof of Proposition~\ref{p:CT-bin-dim-Kron-square}]
The result follows from two inequalities:
$$
f^\nu \. \le \. \binom{n}{\nu_1,\ts \ldots \ts, \ts \nu_r} \. \le \.  \rB(\la,\la',\nu)\ts.
$$
The first inequality is a trivial consequence of
$$
f^\nu \. = \.  \chi^\nu(1) \. \le \, \phi^\nu(1) \. = \. \binom{n}{\nu_1,\ts \ldots \ts, \ts \nu_r}.
$$
The second inequality follows from the following interpretation of the multinomial coefficient.
Start with an \ts $\ell\times m$ matrix $X=(x_{ij})$, where $x_{ij}=1$ if $j\le \la_i$, and $x_{ij}=0$ otherwise.
Consider all binary contingency tables $Y=(y_{ijk}) \in \BCT(\la,\la',\nu)$ which project onto~$X$
along the third (vertical) coordinate.  Because the horizontal margins are equal to $\nu_i$,
the number of such $Y$ is exactly the multinomial coefficient as above.
\end{proof}

\bigskip

\section{Pyramids approach} \label{sec:3CT-pyramids}

\subsection{Lower bound}
Let $\la,\mu,\nu\vdash n$.  A $3$-dimensional binary contingency table \ts
$X=(x_{ijk}) \in \BCT(\la,\mu,\nu)$ \ts is called a \emph{pyramid} if
whenever $x_{ijk}=1$, we also have $x_{pqr}=1$ for all \ts $p\le i$,
\ts $q\le j$, \ts $r\le k$.  Denote by \ts $\PCT(\la,\mu,\nu)$ \ts the set
of pyramids with margins $\la,\mu,\nu$. Finally, let \ts
$\rP(\la,\mu,\nu) := \bigl|\PCT(\la,\mu,\nu)\bigr|$ \ts denote
the number of pyramids.

\begin{thm}[{$\S$\ref{ss:fin-rem-history}}]
\label{t:Kron-3dim-CT-binary-lower}
We have: \. $\rP(\la',\mu',\nu')\ts \le \ts g(\la,\mu,\nu)$.
\end{thm}

\begin{ex}{\rm In notation of Example~\ref{ex:CT-binary-rect-upper-1},
the unique binary table is a pyramid.  This implies that $g(\la,\mu,\nu)=1$
in this case.
}\end{ex}

\begin{ex}{\rm In notation of the Saxl conjecture, we have \ts
$\rP(\rho_\ell,\rho_\ell,\nu) = 0$, unless $\nu=(n)$.  Indeed, suppose
$X=(x_{ijk}) \in \BCT(\rho_\ell,\rho_\ell,\nu)$.  Then
$x_{\ell\ts 1\ts 1} =  x_{1\ts \ell\ts 1} =1$ since the last margins \ts
$a_\ell=b_\ell>0$. Since the first margins $a_1=b_1=\ell$, this implies
that \ts $x_{i\ts 1 \ts 2} = x_{1\ts j \ts 2} = 0$. Thus, \ts
$x_{i\ts j \ts 2} = 0$ for all $1\le i,j\le \ell$.  This implies that
\ts $\nu = (n)$, as desired.  Note also that \ts
$g\bigl(\la,\la,(n)\bigr)=1$ by definition and the symmetry \ts
$g\bigl(\la,\la,(n)\bigr)=g\bigl(\la,(n),\la\bigr)$.

In other words, the above argument and
Proposition~\ref{p:CT-bin-dim-Kron-square} imply that neither
Theorem~\ref{t:Kron-3dim-CT-binary} nor Theorem~\ref{t:Kron-3dim-CT-binary-lower}
are useful for bounding \ts $\rP(\rho_\ell,\rho_\ell,\nu)$ \ts as in the
Saxl conjecture.

In fact, the argument generalizes verbatim for \ts
$\rP(\la,\la',\nu)$, for all $\la\vdash n$.  This shows
that other instances of potential solutions for the
\emph{tensor square conjecture}~\cite{PPV} are also
unreachable with this approach.  }\end{ex}

\medskip

\subsection{Explicit construction} It was shown by Stanley~\cite{Stanley-kron} that
$$
\max_{\la\vdash n} \. \max_{\mu\vdash n} \. \max_{\nu\vdash n}
\,\, g(\la,\mu,\nu) \, = \, \sqrt{n!} \, e^{-O(\sqrt{n})}
$$
In~\cite{PPY}, we refined this to
$$
\frac{f^\la \. f^\mu}{\sqrt{p(n)\. n!}} \, \le \,\max_{\nu \vdash n}  \. g(\la,\mu,\nu)
 \, \le \, \min\bigl\{f^\la,\ts f^\mu\bigr\},
$$
where $p(n)$ is the number of partitions of~$n$.  Stanley's result follows
from an easy asymptotic formula:
$$
\max_{\la \vdash n}  \. f^\la \, = \, \sqrt{n!} \, e^{-O(\sqrt{n})}
$$
(cf.~\cite{VK2}). While we know the asymptotic shape maximizing $f^\la$,
see~\cite{VK1}, we do not know any explicit construction of \ts
$\la,\mu,\nu\vdash n$ \ts which satisfy \ts $g(\la,\mu,\nu) \ge \exp \Omega(n\log n)$.
Here by an \emph{explicit construction} we mean a complexity notion
(there is a deterministic poly-time algorithm for generating the triple),
but in fact we do not have a randomized algorithm either.  In fact, for
the purposes of this section, the naive combinatorial notion of an
``explicit construction'' will suffice.

The current best explicit construction of triples in~\cite[Thm~1.2]{PP2}
has \ts $g(\la,\mu,\nu) = \exp\Theta\bigl(\sqrt{n}\bigr)$, based
on a technical proof using both algebraic and analytic arguments in
the case \ts $\la=\mu=(\ell^\ell)$, for \ts $\nu=(k,k)$, \ts $n=\ell^2=2k$.
See also~\cite{MPP} which gives precise asymptotics in this case.
It is a major challenge to improve upon this (relatively weak) bound.
Below we give an elementary explicit construction of a better lower bound.

\medskip

\begin{thm} \label{t:Kron-explicit}
There is an explicit construction of
\ts $\la,\mu,\nu\vdash n$, such that:
$$g(\la,\mu,\nu) \. = \. \exp \Omega\bigl(n^{2/3}\bigr).$$
\end{thm}

\medskip

To understand this result, observe the following:

\medskip

\begin{prop}\label{p:Kron-explicit-counting}
For some \ts $\la,\mu,\nu \vdash n$, we have
$$
\rP(\la,\mu,\nu) \, = \, \exp \Theta\bigl(n^{2/3}\bigr).
$$
\end{prop}

\begin{proof}
Denote by
$$
p_2(n) \. := \. \sum_{\la,\mu,\nu\vdash n} \. \rP(\la,\mu,\nu)
$$
the number of \emph{plane partitions} of~$n$, see e.g.~\cite{EC2}.
Recall that the number of triples of margins $\la,\mu,\nu\vdash n$ is
$$p(n)^3 \. = \. \exp \Theta(\sqrt{n}), \quad \text{while} \quad
p_2(n) \. = \. \exp \Theta\bigl(n^{2/3}\bigr),
$$
see e.g.~\cite[$\S$VIII.24-25]{FS}. This implies the result.
\end{proof}

\medskip

\begin{lemma}[{\cite[Ex.~3.3]{Val-red}}]
\label{l:Pyr-13}
For \ts
$\al=\be=\ga=(7,4,2) \vdash 13$ \ts we have \ts $\rP(\al,\be,\ga)= 2$.
\end{lemma}

\begin{proof}[Proof of Theorem~\ref{t:Kron-explicit}]
Fix $N=13$, and two distinct $X,X'\in \PCT(\al,\be,\ga)$ for some
$\al=\be=\ga=(7,3,2)\vdash N$ as in the example above.
Denote \ts $\ell = \ell(\al)=3$.

Let $s\ge 1$.  Consider a matrix
$$
Y =(y_{ijk}) \. \in \. \PCT(\theta_{s-1},\theta_{s-1},\theta_{s-1}\bigr)
$$
given by $y_{ijk} =1$ for all $i+j+k \le s+1$, and $y_{ijk} =0$
otherwise, so  $\theta_{s-1} = ( \binom{s}{2}, \binom{s-1}{2},\ldots,\binom{2}{2} )$.
Replace each $1$ by an all-$1$ matrix of size $\ell\times \ell \times \ell$,
each $0$ with $i+j+k=s+2$ with $X$ or~$X'$, and the remaining $0$'s
by an all-$0$ matrix of the same size.  There are \ts $2^{s(s+1)/2}=\exp \Omega(s^2)$ \ts
resulting pyramids which all have the same margins \ts
$\bigl(\la^{(s)},\mu^{(s)},\nu^{(s)}\bigr)$, where \ts
$$
n_s \. := \. \bigl|\la^{(s)}\bigr| \. = \. \bigl|\mu^{(s)}\bigr| \. = \. \bigl|\nu^{(s)}\bigr| \.
= \. \ell^3 \cdot \bigl(1+3+\ldots + s(s-1)/2\bigr) \. + \. N \cdot \binom{s+1}{2}
\, = \, \Theta(s^3)\ts.
$$
By Theorem~\ref{t:Kron-3dim-CT-binary-lower}, this gives a lower bound \ts
$$
g\bigl((\la^{(s)})',(\mu^{(s)})',(\nu^{(s)})'\bigr) \, \ge  \,
\rP\bigl((\la^{(s)}),\mu^{(s)}),(\nu^{(s)})\bigr) \, =
\, \exp \Theta\bigl(n_s^{2/3}\bigr)\.,
$$
as desired.  \end{proof}

\smallskip

\begin{cor} \label{c:Kron-sym-sum}
Let \ts $\cL_n = \{\la\vdash n, \la=\la'\}$.  We have:
$$
\sum_{\la\in \cL_n} \. g(\la,\la,\la) \, = \,
\exp \Omega\bigl(n^{2/3}\bigl)\ts.
$$
\end{cor}

\begin{proof}
A plane partition $A$ is called \emph{totally symmetric} if the
corresponding pyramid has an $S_3$-symmetry, see Case~4
in~\cite{Kra,Stanley-pp} and
\cite[\href{https://oeis.org/A059867}{A059867}]{OEIS}.
Denote by $\cA_n$ the set of totally symmetric
plane partitions of~$n$.
By the symmetry, the margins of $A\in \cA_n$ are triples
$(\la,\la,\la)$, where $\la \in \cL_n$.
From Theorem~\ref{t:Kron-3dim-CT-binary-lower}, we have:
$$
\sum_{\la\in \cL_n} \. g(\la,\la,\la) \, \ge  \,
\sum_{\la\in \cL_n} \. \rP(\la,\la,\la) \, \ge  \, \bigl|\cA_n\bigr|\ts.
$$
It is easy to see by an argument similar to the proof of
Theorem~\ref{t:Kron-explicit} above (or from the explicit
product form GF), that:
$$\bigl|\cA_n\bigr|  \, = \,
\exp \Omega\bigl(n^{2/3}\bigl).
$$
This completes the proof.
\end{proof}

Recall that $|\cL_n|=\exp \Theta\bigl(n^{1/2}\bigl)$,
see e.g.~\cite[\href{https://oeis.org/A000700}{A000700}]{OEIS}.
It was shown in~\cite{BB} that \ts $g(\la,\la,\la) \ge 1$ \ts for all
$\la\in\cL_n$.  This gives only the \ts $\exp \Omega\bigl(n^{1/2}\bigl)$ \ts lower
bound for the LHS in Corollary~\ref{c:Kron-sym-sum}.  Of course, we believe a much
stronger bound lower holds:

\begin{conj} \label{conj:Kron-sym-sum}
Let \ts $\cL_n := \{\la\vdash n, \ts \la=\la'\}$.  We have:
$$
\sum_{\la\in \cL_n} \. g(\la,\la,\la) \, = \,
\exp \. \left[\frac12 \. n \ts \log n \. + \. O(n)\right] \ts.
$$
\end{conj}

\smallskip

In the conjecture, the upper bound follows from
eiquation~\eqref{eq:Kron-dim-upper}.  We refer to \cite[$\S$3]{PPY}
for partial motivation behind this conjecture.

\bigskip

\section{Final remarks and open problems} \label{sec:fin-rem}

\subsection{} \label{ss:fin-rem-history}
The history of
theorems~\ref{t:Kron-3dim-CT},~\ref{t:Kron-3dim-CT-binary}
and~\ref{t:Kron-3dim-CT-binary-lower} is a bit confusing.
As we show, the upper bounds are immediate consequences
of the standard symmetric functions identities, yet
the binary version has been rediscovered imultiple times.
On the one hand, as we show in the first proof of
Theorem~\ref{t:Kron-3dim-CT-binary}, it already follows
from James and Kerber~\cite[Lemma~2.9.16]{JK}, and the
combinatorics of Kostka numbers.  Manivel~\cite{Man}
discusses the same type of inequalities in a different
context, and Vallejo proves Theorem~\ref{t:Kron-3dim-CT-binary-lower}
in a special case~\cite[Cor.~3.5]{Val2}.  The first explicit
statement of Theorem~\ref{t:Kron-3dim-CT-binary}
(up to a restatement as in the second proof of
Theorem~\ref{t:Kron-3dim-CT-binary}),
was given by Vallejo~\cite{Val2}.  In~\cite[Lemma~2.6]{IMW},
these two theorems are included in exactly the same form
as we state. Theorem~\ref{t:Kron-3dim-CT} can be deduced
from Eq.~(6) in~\cite{AV}, although it is not stated in this form.

\subsection{} \label{ss:fin-rem-bound}
Note that in many cases no nontrivial bounds on Kronecker coefficients
are known.  For example, here are the best known bounds in our favorite example:
$$
1 \. \le \. g(\rho_k,\rho_k,\rho_k) \. \le \. f^{\rho_k} \, = \.
\sqrt{n!} \, e^{-O(n)} \ts,
$$
see~\cite{BB,PPY}. Towards the Saxl conjecture, it is known that $g(\rho_k,\rho_k,\mu)>0$
for $\exp\Omega(\sqrt{n})$ partitions $\mu$ with the same principal
hooks as $\rho_k$~\cite[Prop.~4.14]{PPV}.  Same result holds for \ts
$\{\mu \rdom\rho_k\}$~\cite{Ike}.  Of course, the constants implied by the
$\Omega(\cdot)$ notation is  smaller than that in~$p(n)$.
Compare this with a remarkable recent result~\cite[Thm.~B]{BBS}, that
$$
 g(\rho_k,\rho_k,\mu) \. > \. 0 \quad \text{for all} \ f^{\mu} \ \text{odd}.
$$
The number of partitions \ts $\mu\vdash n$ \ts s.t. $f^{\mu}$ \ts is odd is
computed in~\cite{McK}, see also \cite[\href{https://oeis.org/A059867}{A059867}]{OEIS},
and is bounded from above by \ts $\exp O\bigl((\log n)^2\bigr)$.

\subsection{} \label{ss:fin-rem-explicit}
In addition to explicit constructions of Kronecker coefficients,
one can ask similar questions about Kostka numbers and LR--coefficienits.
In fact, this question is trivial for Kostka numbers, since $K(\la,\mu)$
maximizes for $\mu=(1^n)$ and $\la\vdash n$ is of Plancherel shape, cf.~\cite{PPY}.
Recall that $K(\la,\mu)$ is a special case of
LR-coefficients of size $O(n^2)$, see~\cite{PV}.  This gives an \ts
$\exp\Omega(\sqrt{n}\log n)$ \ts lower bound for an explicit
construction of LR--coefficients.  In turn, Murnaghan's theorem
$$g\bigl((N+|\mu|,\nu), \, (N+|\nu|, \mu), \, (N,\la)\bigr) \, = \,
c^\la_{\mu,\nu}\,, \quad \text{for all} \ \ N> |\la| = |\mu|+|\nu|
$$
(see e.g.~\cite{BOR}), would give an explicit construction for
Kronecker coefficients of the same order.  This is weaker than
Theorem~\ref{t:Kron-explicit}.

We expect that one should be able to improve the LR-- and Kronecker
coefficients explicit constructions to $\exp \Theta(n)$ by using
a construction in~\cite[proof of Thm.~4.14]{PPY} based on the
\emph{Knutson--Tao puzzles}.
Of course, that proof is non-explicit and uses a counting argument
similar to the one preceding Lemma~\ref{l:Pyr-13}, but we expect
it to be made explicit in a way similar to the proof of
Theorem~\ref{t:Kron-explicit}.  We intend to return to
this problem in the future.

\subsection{} \label{ss:fin-rem-complexity}  The breakthrough
paper~\cite{IMW} not only proves that the vanishing problem \ts
$g(\la,\mu,\nu) >^? 0$ \ts is $\NP$-hard, it also
proves that computing $g(\la,\mu,\nu)$ is \emph{strongly $\SP$-hard}, i.e.\
$\SP$-hard when the input is in unary.\footnote{C.~Ikenmeyer,
personal communication (2020).}  This would give further evidence
in favor of computing $K(\la,\mu)$ and $\rT(\al,\be)$ being strongly
$\SP$-complete, cf.~\cite{PP1}..

\subsection{} \label{ss:fin-rem-Pyr}
Our original version of Lemma~\ref{l:Pyr-13} was based on
Proposition~\ref{p:Kron-explicit-counting} as follows.
Check that \ts
$p_2(2100) \approx 1.47 \cdot 10^{141}$, while \ts $p(2100)^3 \approx 4.46 \cdot 10^{140}$,
see tables in \cite[\href{https://oeis.org/A000219}{A000219}]{OEIS} and
\cite[\href{https://oeis.org/A000041}{A000041}]{OEIS}, respectively.  Since
$p_2(2100) > p(2100)^3$, we obtain an explicit construction of \ts
$\rP(\al,\be,\ga)\ge 2$ \ts for some $\al,\be,\ga\vdash N=2100$.
In principle, one should be able to avoid even this calculation, and obtain a bound
$p_2(n) > p(n)^3$ for an explicit $n$ by using tight asymptotic estimates.
Unfortunately, bounds on the error term for $p(n)$, see e.g.~\cite{DP},
are lacking for $p_2(n)$, cf.~\cite{GP}.

The value $N=2100$ is much larger than $N=13$ in Lemma~\ref{l:Pyr-13}.
We first learned of the example $\la=(7,4,2)$ from John Machacek.\footnote{See \url{https://mathoverflow.net/questions/351376/plane-partitions-with-equal-margins}\ts.}
In fact, Gjergji Zaimi showed that $n=13$ is the smallest possible (ibid.)  Vallejo later
informed us that he published this example in~\cite{Val-red}.
To understand and generalize this example, take a \emph{cyclically symmetric}
but not totally symmetric plane partition~$A$ (see~\cite{Kra,Stanley-pp}, Case~3),
with margins $(\la,\la,\la)$, s.t. $\la=\la'$. Then $A'\ne A$ but both plane partitions
have the same margins, implying that \ts $g(\la,\la,\la)\ge 2$.

\subsection{} \label{ss:fin-rem-MO} In the context of $\S$\ref{sec:3CT-pyramids},
we believe in the following claims.

\begin{conj}\label{c:pyramis-many}
Denote by $a(n)$ and $b(n)$ the number of triples $(\la,\mu,\nu)$, $\la,\mu,\nu\vdash n$,
s.t. $\rB(\la,\mu,\nu)\ge 1$ \ts and \ts $\rP(\la,\mu,\nu)\ge 1$, respectively.
Then:
$$
\frac{a(n)}{p(n)^3} \. \to \. 1 \quad \text{and} \quad  \frac{b(n)}{p(n)^3} \. \to \. 0 \ \ \text{as} \ \ n\to \infty\ts.
$$
\end{conj}

On the other hand, we believe that the Kronecker coefficients are
non-vanishing a.s.

\begin{conj}\label{c:Kron-many}
Denote by $c(n)$ the number of triples $(\la,\mu,\nu)$, such that
\ts $\la,\mu,\nu\vdash n$ \ts and \ts $g(\la,\mu,\nu)\ge 1$.
Then:
$$
\frac{c(n)}{p(n)^3} \. \to \. 1 \ \ \text{as} \ \ n\to \infty\ts.
$$
\end{conj}

\smallskip

These conjectures are motivated by Conjecture~8.3 in~\cite{PPV}
which states that $\ts \chi^\la[\mu]\ne 0\ts$~a.s.
In the opposite direction, it is known that the Kostka numbers
$K(\la,\mu)=0$~a.s. This is equivalent to the former
\emph{Wilf's conjecture} that the probability \ts $P(\la \dom\mu)\to 0$
\ts for uniform random \ts $\la,\mu\vdash n$.  This conjecture was
resolved by Pittel
in~\cite{Pit1} (see also recent~\cite{Pit2} for effective bounds).

\subsection{} \label{ss:fin-rem-pp-shape}
In the context of Section~\ref{sec:3CT-pyramids}, it is
worth noting that the limit shape for plane partitions is
well known~\cite{CK}, see also~\cite{Ok}.  The limit shape is totally symmetric
with margin curves corresponding to partitions $\la\vdash n$
with \ts
$f^\la = \sqrt{n!} \. e^{-O(n)}$.  This does not immediately
suggests that Conjecture~\ref{conj:Kron-sym-sum} holds, but
only that there is a large gap between the upper and lower
bounds.

\vskip.76cm

\subsection*{Acknowledgements}
We are grateful to Christine Bessenrodt, Sam Dittmer, Vadim Gorin,
Christian Krattenthaler, Han Lyu, Alejandro Morales,
Fedya Petrov and Damir Yeliussizov for
interesting discussions and useful comments.
We thank Brian Hopkins, John Machacek,
David Speyer and Gjergji Zaimi for their very helpful answers
to our MO question (see~$\S$\ref{ss:fin-rem-Pyr}), and to
Mercedes Rosas for telling us about the forthcoming paper~\cite{BR}.
We are especially thankful to Ernesto Vallejo for many helpful
historical and reference comments, and to Sasha Barvinok
for his help and guidance over the years on the ever-remarkable subject of
contingency tables.  The second author benefited from the cold Moscow
weather which pushed the first author to stay inside and work on the project.
Both authors were partially supported by the NSF.

\end{document}